\documentclass[10pt,a4paper]{article}
\usepackage{amsfonts,amssymb,amsmath}
\usepackage{theorem,bm}

  \theoremstyle{change}
  {\theorembodyfont{\rm}
  \newtheorem{defi}{Definition}[section]
  \newtheorem{rem}[defi]{Remark.}
  
   }

  {\theorembodyfont{\it}
  \newtheorem{lem}[defi]{Lemma.}
  \newtheorem{prop}[defi]{Proposition.}
  \newtheorem{thm}[defi]{Theorem.}
  \newtheorem{cor}[defi]{Corollary.}  }

\newenvironment{nrtxt}{\begin{trivlist}\refstepcounter{defi}\item\textbf{\arabic{section}.\arabic{defi}.}}
{\end{trivlist}}

\newenvironment{txt}{\begin{trivlist}\item}{\end{trivlist}}

\newcommand{\gerade}[1]{{\rm#1}}

\newcommand{\PP}{{\mathbb P}}
\newcommand{\QQ}{{\mathbb Q}}
\newcommand{\ZZ}{{\mathbb Z}}

\newcommand{\cC}{{\mathcal C}}

\newcommand{\cG}{{\mathcal G}}
\newcommand{\cL}{{\mathcal L}}
\newcommand{\cP}{{\mathcal P}}

\newcommand{\cS}{{\mathcal S}}

\newcommand{\fM}{{\mathfrak M}}

\newcommand{\vu}{{\bm{u}}}
\newcommand{\vv}{{\bm{v}}}
\newcommand{\vw}{{\bm{w}}}
\newcommand{\vM}{{\bm{M}}}
\newcommand{\vN}{{\bm{N}}}
\newcommand{\vP}{{\bm{P}}}
\newcommand{\vQ}{{\bm{Q}}}
\newcommand{\vU}{{\bm{U}}}
\newcommand{\vX}{{\bm{X}}}
\newcommand{\vY}{{\bm{Y}}}

\newcommand{\Aut}{{\mathrm{Aut}}}

\newcommand{\End}{{\mathrm{End}}}

\newcommand{\diag}{{\mathrm{diag}}}

\DeclareMathOperator{\rad}{rad}
\DeclareMathOperator{\dist}{dist}

\newcommand{\T}{{\mathrm{T}}}

\newcommand{\dis}{\mathbin{\scriptstyle\triangle}}

\newcommand{\notdis}{\mathbin{\not\scriptstyle\triangle}}

\newcommand{\notpar}{\mathbin{\not{}\hspace{-0.35ex}{\parallel}\hspace{0.35ex}}}

\newcommand{\eps}{{\varepsilon}}

\newcommand{\GL}{{\mathrm{GL}}}

\let\phi=\varphi

\let\theta=\vartheta

\newcommand{\DelimArray}[4]{\left#1\begin{array}{*{#3}{c}}#4\end{array}\right#2}
\newcommand{\SDelimArray}[4]{\hbox{\scriptsize\arraycolsep=.5\arraycolsep
  $\left#1\begin{array}{*{#3}{c}}#4\end{array}\right#2$}}

\newcommand{\Mat}{\DelimArray()}
\newcommand{\SMat}{\SDelimArray()}

\newenvironment{proof}
    {\begin{trivlist} \item {\it Proof.}} 
    {\mbox{~}\hfill$\square$\end{trivlist}}

\begin{document}

\title{On distant-isomorphisms of projective lines}

\author{Andrea Blunck \and Hans Havlicek}

\date{}

\maketitle

\begin{abstract}\noindent
We determine all distant-iso\-mor\-phisms between projective lines over
semilocal rings. In particular, for those semisimple rings that do not have a
simple component which is isomorphic to a field, every distant isomorphism
arises from a Jordan isomorphism of rings and a projectivity. We show this by
virtue of a one-one correspondence linking the projective line over a
semisimple ring with a Segre product of Grassmann spaces.


\end{abstract}

%
\section{Introduction.}\label{se:intro}

\begin{nrtxt}
The projective line over a (skew) field is often said to have no intrinsic
structure. However, if we consider a ring $R$ other than a field, then the
projective line $\PP(R)$ carries a non-trivial relation ``distant''.
Intuitively, points $p$ and $q$ of $\PP(R)$ are \emph{distant\/} (in symbols:
$p\dis q$) if $p$ and $q$ span the entire projective line, whereas
non-distant points are ``too close together'' in order to share this
property. Therefore non-distant points are also called \emph{neighbouring}.
If $R$ is a field, then $p\dis q$ just means that $p\neq q$, whence in this
case the distant relation does not deserve our interest, and $\PP(R)$ indeed
has no intrinsic structure.

For rings that are not fields, however, the following question is
interesting: Which bijections between projective lines over rings $R$ and
$R'$ map distant points to distant points and non-distant points to
non-distant points? Every mapping of this kind will be called a
\emph{distant-iso\-mor\-phism}. It seems that no attention has been paid to
this question so far. One reason could be that we cannot expect an algebraic
description of distant-iso\-mor\-phisms for some classes of rings like
fields, direct products of fields, or local rings. Here, loosely speaking,
``algebraic description'' means that the mapping under consideration is the
product of a bijection $\PP(R)\to\PP(R')$ which stems from a Jordan
isomorphism $R\to R'$ and a projectivity of $\PP(R')$.

In \cite{blu+h-00c} it was shown that such an algebraic description of
distant-iso\-mor\-phisms is possible if both $R$ and $R'$ are rings of $2\times
2$ matrices over fields. In the present note we extend this result to the case
when both $R$ and $R'$ are direct products of matrix rings over fields,
provided that no ring of $1\times 1$ matrices is present in these products
(Theorem~\ref{thm:alg.darst.prod}).

In order to reach this goal we bring together several concepts and results.
Some basic facts about the projective line $\PP(R)$ are presented in
Section~\ref{se:morphismen}. In Section~\ref{se:adjazenz} we recall the
definition and some properties of a \emph{parallelism\/} (written
as~$\parallel$) on $\PP(R)$; it was introduced in \cite{blu+h-03a}, and it
reflects the Jacobson radical of $R$ in geometric terms. This
relation~$\parallel$ is, by its definition, invariant under
distant-iso\-mor\-phisms. At this point we have to emphasize that many
authors, like Herzer in \cite{herz-95}, use the symbol~$\parallel$ and the
phrase ``parallel points'' in a different meaning, namely for what we simply
call non-distant points. Next, we introduce an \emph{adjacency relation\/}
(written as~$\sim$) on $\PP(R)$. This notion comes from the geometry of
Grassmann spaces. It is well known that the projective line over a ring of
matrices with entries from a field is in one-one correspondence with the
points of some Grassmann space. In a very general form, this is a result
about the projective line over the endomorphism ring of a (possibly
infinite-dimensional) vector space; see \cite{blunck-99}. Clearly, we could
use this one-one correspondence in order to define an adjacency relation on
$\PP(R)$ as the preimage of the adjacency relation on the associated
Grassmann space. However, we use an intrinsic definition of adjacency on
$\PP(R)$ in terms of the distant relation; the idea for this definition is
taken from \cite{blu+h-02z}.

In Section~\ref{se:isomorphismen} we show that one cannot expect an algebraic
description of distant-iso\-mor\-phisms for rings with a non-zero Jacobson
radical, because for these rings one has a non-trivial parallelism, and every
permutation of $\PP(R)$ that maps each point to a parallel one is a
distant-auto\-mor\-phism. Then, in Section~\ref{se:endo-ringe}, we establish
our main result in a first step for matrix rings over fields
(Theorem~\ref{thm:alg.darst}); this is generalized in Section \ref{se:produkte}
to \emph{semisimple\/} rings, i.e. direct products of finitely many matrix
rings over fields (Theorem~\ref{thm:alg.darst.prod}). Altogether, the results
from Sections~\ref{se:isomorphismen}, \ref{se:endo-ringe}, and
\ref{se:produkte} give a complete description of the distant-iso\-mor\-phisms
for \emph{semilocal\/} rings, i.e. rings which are semisimple modulo their
Jacobson radical (Corollary~\ref{cor:iso}).

In Sections~\ref{se:endo-ringe} and \ref{se:produkte} we use Grassmann spaces
and their Segre products to represent the projective lines over semisimple
rings in terms of partial linear spaces. In this setting,
distant-iso\-mor\-phisms correspond to collineations of these spaces
(Propositions~\ref{prop:iso.endo} and \ref{prop:iso.prod}), and the adjacency
relations are the key to all this. Then we can use a recent result by Naumowicz
and Pra\.zmowski (see \cite{nau+p-01}) in order to accomplish our work. As a
matter of fact, we need their result in a slightly generalized form, which is
presented as an Appendix to this article.

Alternatively, our results can also be viewed as a generalization of Chow's
theorem \cite{chow-49} on adjacency preserving bijections of Grassmann spaces
to Segre products of such spaces. Cf.\ also \cite{benz-92}, \cite{huang-98},
and \cite{wan-96}.

There is a wealth of literature on harmonic mappings for projective lines
over rings. A survey is given by Bartolone and Bartolozzi \cite{bart+b-85},
and Lashkhi \cite{lash-97}; see also \cite{blu+h-03}. As every harmonic
mapping takes distant points to distant points, our results may also be
viewed as a contribution to this topic. However, this connection will be
discussed elsewhere.
\end{nrtxt}

\begin{nrtxt}
Throughout this paper we shall only consider associative rings with a unit
element $1\neq 0$, which is preserved by (anti-)homomorphisms, inherited by
subrings, and acts unitally on modules.  The group of invertible elements
(units) of a ring $R$ is denoted by $R^*$. Moreover, by a \emph{field\/} we
always mean a not necessarily commutative field. We refer to \cite{lam-91}
for those notions and results from ring theory which are used in the text
without reference.
\end{nrtxt}

\section{Distant-morphisms.}\label{se:morphismen}

\begin{nrtxt}
Consider a ring $R$ and the free left $R$-module $R^2$. The \emph{projective
line over\/} $R$ is the orbit of the free cyclic submodule $R(1,0)$ under the
natural right action of the group $\GL_2(R)$ of invertible $2\times 2$ matrices
with entries in $R$. In other words, $\PP(R)$ is the set of all $p\leq R^2$
such that $p=R(a,b)$, where $(a,b)$ is the first row of an invertible matrix.
See \cite[p.~785]{herz-95}. If also $(c,d)$ is the first row of an invertible
matrix,  then $R(a,b)=R(c,d)$ if, and only if, there is a unit $u\in R^*$ with
$(c,d)=u(a,b)$ (see \cite[Prop.~2.1]{blu+h-00b}).

Let $\big((a,b), (c,d)\big)$ be a basis of $R^2$ or, said differently, let
$\SMat2{a&b\\c&d}\in\GL_2(R)$. Then the points $p=R(a,b)$ and $q= R(c,d)$ are
called \emph{distant}, and we write $p\dis q$. Obviously, $\dis$ is an
anti-reflexive and symmetric relation on $\PP(R)$.

In the special case that $R$ is a  field, $\PP(R)$ is  the usual projective
line over $R$, and $\dis$ is the relation $\ne$.

The projective line $\PP(R)$ together with the relation $\dis$ can also be
seen as a graph (the \emph{distant graph\/}) with vertex set $\PP(R)$ and two
vertices joined by an edge if, and only if, they are distant. See
\cite{blu+h-01a} for graph theoretic properties of $(\PP(R),\dis)$.

For each point $p\in\PP(R)$ we let $\dis(p)$ be the neighbourhood of $p$ in
the distant graph, i.e.\ the set of all points distant to $p$.
\end{nrtxt}

\begin{nrtxt}
We are interested in mappings preserving the  distant relation. Let $R$ and
$R'$ be rings, and let $\phi:\PP(R)\to\PP(R')$ be a mapping with
\begin{equation}\label{def:dis-morphism}
  \forall\,p,q\in\PP(R): p\dis q \Rightarrow p^\phi\dis  q^\phi,
\end{equation}
where, by abuse of notation, the distant relation on $\PP(R')$ is denoted by
$\dis$ rather than $\dis'$. Then we call $\phi$ a \emph{distant-mor\-phism\/}
(or \emph{$\dis$-mor\-phism}). So the $\dis$-mor\-phisms are exactly the
homomorphisms of our distant graphs. If $\phi$ is a bijective $\dis$-mor\-phism
and also $\phi^{-1}$ is a $\dis$-mor\-phism then, as usual, we call $\phi$ a
\emph{$\dis$-iso\-mor\-phism}.
\end{nrtxt}

\begin{rem}
It is worth noting that a bijective $\dis$-mor\-phism need not be a
$\dis$-iso\-mor\-phism. Indeed, let $\ZZ$ be the ring of integers and let
$\QQ$ be the field of rational numbers. A pair $(a,b)\in\ZZ^2$ is the first
row of an invertible matrix if, and only if, it is unimodular, i.e., if there
are integers $x$ and $y$ such that $ax+by=1$ \cite[Prop.~1.4.1]{herz-95}. On
the other hand, each point of $\PP(\QQ)$ can be written in the form
$\QQ(a,b)$ with an unimodular pair $(a,b)\in\ZZ^2$. In either case two
unimodular pairs of integers yield the same point precisely when they are
proportional by a factor $\pm 1$. Therefore the mapping $\phi:\PP(\ZZ) \to
\PP(\QQ) : \ZZ(a,b)\mapsto \QQ(a,b)$ is well defined and bijective. Condition
(\ref{def:dis-morphism}) is satisfied, as here $p^\phi\dis q^\phi$ just means
$p^\phi\neq q^\phi$. However, by $\ZZ(1,0)\notdis\ZZ(1,2)$ and
$\QQ(1,0)\neq\QQ(1,2)$, the bijection $\phi$ is not a $\dis$-iso\-mor\-phism.

Cf.\ Remark~1 in \cite[p.~359]{bart+b-85} for another example of such a
bijective $\dis$-mor\-phism.
\end{rem}

\begin{nrtxt}\label{:beispiele}
There are several classes of examples of $\dis$-mor\-phisms:

(a) By definition, the group $\GL_2(R)$ acts (transitively) on $\PP(R)$ and
leaves $\dis$ invariant. So, for each $\gamma\in\GL_2(R)$, the induced
\emph{projectivity\/} $\widetilde\gamma:\PP(R)\to \PP(R):R(a,b)\mapsto
R\big((a,b)^\gamma\big)$ is a $\dis$-auto\-mor\-phism.

(b) Each ring homomorphism $\alpha:R\to R'$ gives rise to a mapping
\begin{equation}\label{def:tilde.alpha}
\widetilde\alpha:\PP(R)\to\PP(R'):R(a,b)\mapsto R'(a^\alpha,b^\alpha).
\end{equation}
One can easily check that  $\widetilde\alpha$ maps distant points to distant
points and hence is a $\dis$-mor\-phism. In particular, if $\alpha$ is an
isomorphism of rings, then $\widetilde{\alpha}$ is a $\dis$-iso\-mor\-phism.

(c) Let $\alpha:R\to R'$ be an anti-homo\-mor\-phism of rings. In general,
$\widetilde\alpha$ as in (\ref{def:tilde.alpha}) is not well defined any
more. Instead, we may proceed as follows. For each point $p$ there is a
regular matrix $M$ with first row $(a,b)$, say, and $p=R(a,b)$. Let
$(v,w)^\T$, where $\T$ denotes transposition of matrices, be the second
column of $M^{-1}$. Then we define
\begin{equation}\label{def:tilde.alpha.anti}
\widetilde\alpha: \PP(R)\to\PP(R'):p\mapsto R'(-w^\alpha,v^\alpha).
\end{equation}
This is a well defined $\dis$-mor\-phism, as follows from the proof of
\cite[Thm.~5.2]{blu+h-01b} and \cite[Rem.~5.4]{blu+h-01b}, where all this is
studied in the context of chain geometries and dual pairs of $R$-modules.
(For finite-dimensional algebras this result is due to Herzer
\cite[Prop.~3.3]{herz-87a}.)
\par
Suppose that $\alpha$ is also a homomorphism, whence the ring $R^\alpha$ is
commutative. So we may apply the cofactor method to invert matrices with
coefficients in $R^\alpha$. By calculating the second column of
$(M^{-1})^{\alpha}=(M^\alpha)^{-1}$ in two ways, we get
$(v^\alpha,w^\alpha)^\T=(\det M^\alpha)^{-1}(-b^\alpha,a^\alpha)^\T$. Thus, the
definitions in (\ref{def:tilde.alpha}) and (\ref{def:tilde.alpha.anti})
coincide, and it is unambiguous to use the same symbol $\widetilde\alpha$ in
either case.

(d) Now consider a \emph{Jordan homomorphism\/} $R\to R'$, i.e.\ an additive
mapping $\alpha$ satisfying $1^\alpha=1\in R'$ and $(aba)^\alpha=a^\alpha
b^\alpha a^\alpha$ for all $a,b\in R$. Of course, each homomorphism or
anti-homo\-mor\-phism is also a Jordan homomorphism. See
Corollary~\ref{cor:jordan-prod} and \cite[3.8]{blu+h-03} for examples of Jordan
homomorphisms that are neither homomorphisms nor anti-homo\-mor\-phisms.

In \cite[Thm.~4.4~(b)]{blu+h-03} it is shown that each Jordan homomorphism
gives rise to \emph{at least one\/} $\dis$-mor\-phism. However, the
definition itself is rather involved, and we do not need it here. Instead, we
present a special case, due to Bartolone.

Let $R$ be a ring of \emph{stable rank\/ $2$} (for the definition see
\cite[p.~1039]{veld-95}). Then
\begin{equation}\label{eq:bart1}
\PP(R)=\{R(ab-1,a)\mid a,b\in R\}.
\end{equation}
If, moreover, $R'$ is any ring and $\alpha:R\to R'$ is a Jordan homomorphism,
then
\begin{equation}\label{eq:bart2}
\widetilde\alpha:R(ab-1,a)\mapsto R'(a^\alpha b^\alpha-1, a^\alpha)
\end{equation}
is a $\dis$-mor\-phism $\PP(R)\to \PP(R')$. See  \cite[Thm.~(2.4)]{bart-89}.
If $\alpha$ is a homomorphism or an anti-homo\-mor\-phism, then
$\widetilde\alpha$ coincides with the mapping of (\ref{def:tilde.alpha}) or
(\ref{def:tilde.alpha.anti}), respectively. In case of a homomorphism this is
obvious, otherwise we refer to \cite[Rem.~5.4]{blu+h-01b}.
\end{nrtxt}

\section{Adjacency.}\label{se:adjazenz}

\begin{nrtxt}\label{para:parallel+adj}
The distant relation on $\PP(R)$ yields two further binary relations. First we
recall the following definition from \cite{blu+h-03a}: Consider $p,q\in\PP(R)$.
We say that $p$ and $q$ are \emph{parallel\/} if
\begin{equation}\label{def:parallel}
    \dis(p)\subseteq \dis(q).
\end{equation}
In this case we write $p\parallel q$. In  \cite[Cor.~2.3]{blu+h-03a} it is
shown that~$\parallel$ (called \emph{radical parallelism\/} there) is an
equivalence relation. So
\begin{equation}\label{eq:parallel.symmetrie}
  p\parallel q \Leftrightarrow \dis(p)=\dis(q).
\end{equation}

For $p,q\in\PP(R)$ we always have
\begin{equation}\label{eq:parallel.notdis}
    p\parallel q \Rightarrow p\notdis q,
\end{equation}
because otherwise $q\in\dis(p)\subseteq \dis(q)$ in contradiction to $q\notdis
q$.

The \emph{Jacobson radical\/} $\rad R$ of the ring $R$ is the intersection of
all maximal left (or all maximal right) ideals of $R$. It is a two-sided ideal
of $R$, whence one can consider the canonical epimorphism
\begin{equation}\label{eq:canon.epi}
   \pi:R\to \overline R:=R/\rad R: a\mapsto \overline a:= a+ \rad R.
\end{equation}
According to (\ref{def:tilde.alpha}), we obtain the associated
$\dis$-mor\-phism
\begin{equation}\label{eq:canon.Pepi}
  \widetilde\pi:\PP(R)\to\PP(\overline R): p\mapsto \overline
p:=p^{\widetilde\pi}
\end{equation}
which is surjective (see \cite[Prop.~3.5]{blu+h-00b}). By
\cite[Thm.~2.2]{blu+h-03a}, for $p,q\in \PP(R)$ we have
\begin{equation}\label{eq:parallel}
p\parallel q \Leftrightarrow \overline p = \overline q.
\end{equation}
\end{nrtxt}

\begin{nrtxt}
We call the points $p,q\in\PP(R)$ \emph{adjacent}, and write $p\sim q$, if
\begin{equation}\label{def:adjacent}
\exists\, r\in\PP(R): r \notpar p,q \mbox{ and } {\dis(r)}\subseteq
\dis(p)\cup\dis(q).
\end{equation}
In this situation we also say that $p$ is \emph{adjacent\/} to $q$ \emph{via}
$r$.

Obviously, the relation $\sim$ is symmetric and anti-reflexive; we even have
$p\sim q \Rightarrow p \notpar q$, since otherwise, if $p\sim q$ via $r$, we
had $\dis(r)\subseteq\dis(p)\cup\dis(q)=\dis(p)$ and hence $r\parallel p$,
which is not allowed by (\ref{def:adjacent}). A description of the relation
$\sim$ for arbitrary rings is beyond the scope of this article. However, in
\ref{:adjacent} and Proposition~\ref{prop:adjazent} we shall show where this
definition has its origin.

The next lemma says how  $\dis$ and $\sim$ behave on parallel classes.
\end{nrtxt}

\begin{lem}\label{lem:mod.parallel}
Let $p_1,q_1,r_1,p_2,q_2,r_2\in\PP(R)$ with $p_1\parallel p_2$, $q_1\parallel
q_2$, and $r_1\parallel r_2$. Then

\vspace{-0.9ex}\begin{enumerate}\itemsep0pt
 \item\label{lem:mod.parallel.a}
$p_1\dis q_1\Leftrightarrow p_2\dis q_2$,
 \item\label{lem:mod.parallel.b}
$p_1\sim q_1$ via $r_1$ $\Leftrightarrow $ $p_2\sim q_2$ via $r_2$.
\end{enumerate}
\end{lem}
\begin{proof}
(\ref{lem:mod.parallel.a}): Since $p_1\parallel p_2$, we get from $p_1\dis q_1$
that $q_1\dis p_2$. Then, since $q_1\parallel q_2$, we get from $q_1\dis p_2$
that $p_2\dis q_2$.

(\ref{lem:mod.parallel.b}): Let $p_1\sim q_1$ via $r_1$, i.e., $r_1\notpar
p_1,q_1$ and $\dis(r_1)\subseteq \dis(p_1)\cup \dis(q_1)$. Transitivity
of~$\parallel$ yields $r_2\notpar p_2,q_2$. Moreover,
$\dis(r_2)=\dis(r_1)\subseteq \dis(p_1)\cup\dis(q_1)=\dis(p_2)\cup\dis(q_2)$.
\end{proof}

\begin{txt}
This lemma means that the relations $\dis$ and $\sim$ are well defined on the
set of parallel classes of $\PP(R)$. By (\ref{eq:parallel}), this can also be
formulated in terms of the surjective  mapping $\widetilde\pi$ onto the
projective line over $\overline R=R/\rad R$\/:
\end{txt}

\begin{cor}\label{cor:Rquer}
Let $\widetilde\pi:\PP(R)\to \PP(\overline R):p\mapsto\overline p$ be as in
\gerade{(\ref{eq:canon.Pepi})}. Then for all $p,q,r\in\PP(R)$ we have

\vspace{-0.9ex}\begin{enumerate}\itemsep0pt
 \item\label{Rquer.a}
$p\dis q\Leftrightarrow \overline p\dis \overline q$,
 \item\label{Rquer.b}
$p\sim q$ via $r$ $\Leftrightarrow $ $\overline p\sim \overline q$ via
$\overline r$.
\end{enumerate}
\end{cor}
\begin{txt}
Assertion (\ref{Rquer.a}) can also be shown algebraically, see
\cite[Props.~3.1, 3.2]{blu+h-00b}.
\end{txt}

\begin{rem}\label{rem:local}
As an example, we study the three relations $\dis$, $\parallel$, and~$\sim$ for
local rings, i.e.\ rings $R$ where the set $R\setminus R^*$ of non-units is an
ideal. Then $R\setminus R^*$ is the unique  maximal ideal of $R$ and coincides
with the Jacobson radical.

For an arbitrary ring $R$ the following characterizations were established in
\cite[Prop.~2.4.1]{herz-95} and \cite[Thm.~2.5]{blu+h-03a}: $R$ is local if,
and only if, $\notdis$ is an equivalence relation on $\PP(R)$, which in turn
is equivalent to
\begin{equation}
 \forall\,p,q\in\PP(R): p\notdis
q\Leftrightarrow p\parallel q.
\end{equation}
We claim that for a local ring $R$ we have
\begin{equation}
\forall\,p,q\in\PP(R): p\dis q \Leftrightarrow p\sim q.
\end{equation}
In fact, we infer from Corollary~\ref{cor:Rquer} (\ref{Rquer.a}) that $p\dis
q\Leftrightarrow \overline p\dis\overline q$. Since $\overline R$ is a field,
we can rewrite the last condition as $\overline p\neq \overline q$. Taking into
account that $\PP(\overline R)$ contains at least three different points, we
see that $\overline p\neq \overline q\Leftrightarrow \overline p \sim \overline
q$ (via any point $\overline r\neq \overline p,\overline q$). By
Corollary~\ref{cor:Rquer} (\ref{Rquer.b}), this is equivalent to $p\sim q$.
\end{rem}

\begin{nrtxt}
Let us recall the concept of a \emph{partial linear space\/} (or
\emph{semilinear space\/} or \emph{partial line space\/}): This is a point-line
geometry $\fM=(\cP, \cL)$, where any two distinct points are joined by at most
one line, and every line contains at least two distinct points. Note that we do
not make any richness conditions on $\cP$ or $\cL$, as some authors do. If
points are on a common line then they are said to be \emph{collinear}. A
\emph{strong subspace\/} (or \emph{linear subspace\/}) of $\fM$ is a set
$\cS\subseteq\cP$ of mutually collinear points such that $\cS$ is closed under
lines. Following \cite{nau+p-01}, a partial linear space $\fM$ is called
\emph{strongly connected\/} if for every point $p\in\cP$ and every strong
subspace $\cS\subseteq\cP$, with $\#\,\cS>1$, there is a finite sequence
$\cS_0,\cS_1,\ldots,\cS_m$ of strong subspaces such that $\cS=\cS_0$,
$p\in\cS_{m}$ and $\#(\cS_{i-1}\cap\cS_i)\geq 2$ for all
$i\in\{1,2,\ldots,m\}$.
\end{nrtxt}

\begin{nrtxt}\label{:adjacent}
Let $\vU$ be a left vector space over a field~$K$, $\dim\vU\geq 1$, and let
$R=\End_K(\vU)$ be its endomorphism ring. We restrict ourselves to the case
when $n:=\dim\vU$ is finite, even though most of the subsequent results can be
formulated in such a way that they remain valid also in case of infinite
dimension.

The following is taken from \cite[Thm.~2.4]{blunck-99}, whereas in
\cite{blu+h-00b} the viewpoint of projective geometry was adopted: Let $\cG$ be
the set of all $n$-dimensional subspaces of the vector space $\vU\times \vU$.
Then
\begin{equation}\label{def:Psi}
\Psi:\PP(R)\to \cG: R(a,b)\mapsto \vU^{(a,b)}:= \{(\vu^a,\vu^b)\mid \vu\in
\vU\}
\end{equation}
is a well defined bijection mapping distant points of $\PP(R)$ to
complementary subspaces in $\cG$ and non-distant points to non-complementary
subspaces. Moreover, the groups $\GL_2(R)$ and $\Aut_K(\vU\times \vU)$ are
isomorphic, and their actions on $\PP(R)$ and on $\cG$, respectively, are
equivalent via $\Psi$.

On $\cG$ there is a binary relation
\begin{equation}\label{def:adjacent.G}
 \vP\sim \vQ:\Leftrightarrow \dim(\vP\cap\vQ)=n-1
 \Leftrightarrow \dim(\vP+\vQ)=n+1,
\end{equation}
where $\sim$ is to be read as \emph{adjacent}. It has been studied by many
authors. See, for example, \cite{chow-49}, \cite{huang-98}, and \cite{wan-96}.
The graph with vertex set $\cG$ and two vertices $\vP,\vQ$ joined by an edge
if, and only if, $\vP\sim \vQ$ is called the \emph{Grassmann graph\/} on $\cG$.
Compare, among others, \cite{brouwer+cohen+neu-89}, \cite{fu+huang-94}, and
\cite{numata-90}.

The set $\cG$ is the point set of the \emph{Grassmann space\/} $\fM=(\cG,\cL)$,
with line set $\cL$ consisting of all \emph{pencils\/}
\begin{equation}
\cG[\vM,\vN]:=\{\vX\in\cG\mid \vM\subset\vX\subset\vN\},
\end{equation}
where $\dim \vM=n-1$, $\dim \vN=n+1$, and $\vM\subset \vN$. The Grassmann space
$\fM$ has at least one line and at least three points on every line. It is a
strongly connected partial linear space; see \cite[Prop.~1.13]{nau+p-01}, where
this is shown for a wider class of Grassmann spaces. (In \cite{nau+p-01} the
term \emph{spaces of pencils\/} is used instead. If $\vU$ has infinite
dimension then there are points which cannot be joined by a (finite) polygonal
path, i.e., the partial linear space $\fM$ is not even connected.)

The following result is crucial for the entire paper:
\end{nrtxt}

\begin{prop}\label{prop:adjazent}
Let $R=\End_K(\vU)$ be the endomorphism ring of an $n$-dimensional left vector
space $\vU$ over a field $K$, $1\leq n<\infty$, let $\fM$ be the associated
Grassmann space, and let $\Psi:\PP(R)\to \cG$ be given as in
\gerade{(\ref{def:Psi})}. Then the following statements hold for all $p$, $q$,
and $r\in\PP(R)$\gerade:

\vspace{-0.9ex}\begin{enumerate}\itemsep0pt
 \item\label{adjazent.a}
    $p\sim q \Leftrightarrow p^\Psi \sim q^\Psi \Leftrightarrow p^\Psi$ and
    $q^\Psi$  are distinct collinear points of $\fM$.

 \item\label{adjazent.b}
$p\sim q$  via $r \Leftrightarrow p^\Psi$, $q^\Psi$, and $r^\Psi$ are distinct
collinear points of $\fM$.
\end{enumerate}
\end{prop}
\begin{proof}
Note that in the present situation the condition $r\notpar p,q$ in
(\ref{def:adjacent}) reduces to $r\neq p,q$, because $R=\End_K(\vU)$ is simple,
whence $\rad R=\{0\}$, and parallelity means equality by (\ref{eq:parallel}).
Now the first equivalence in (\ref{adjazent.a}) is an immediate consequence of
\cite[Thm.~3.2]{blu+h-02z}. Furthermore, $p^\Psi\sim q^\Psi$ just means that
$p^\Psi\neq q^\Psi$ both belong to a line of $\fM$, namely $\cG[p^\Psi\cap
q^\Psi,p^\Psi+q^\Psi]$.

For the assertion in (\ref{adjazent.b}) we refer to the proof of
\cite[Thm.~3.2]{blu+h-02z}.
\end{proof}

\section{Distant-isomorphisms.}\label{se:isomorphismen}

\begin{nrtxt}
After having seen in \ref{:beispiele} how to construct $\dis$-mor\-phisms from
given algebraic mappings, we now aim at a description of all
$\dis$-iso\-mor\-phisms. Clearly, every $\dis$-iso\-mor\-phism is also a
$\parallel$-iso\-mor\-phism and a $\sim$-iso\-mor\-phism, where
$\parallel$-mor\-phisms and $\sim$-mor\-phisms are defined like
$\dis$-mor\-phisms.
\end{nrtxt}

\begin{rem}
Note that an injective $\dis$-mor\-phism $\phi$ need not be a
$\parallel$-mor\-phism. Take, e.g., the ring $R=K(\eps)$ (with $\eps^2=0$) of
\emph{dual numbers\/} over a field $K$. The right regular representation of
$R$ is the monomorphism $\alpha:R\to\End_K(R): a\mapsto (x\mapsto xa)$. The
associated mapping $\widetilde\alpha$ is injective, whence it maps the
parallel points $R(1,0)$ and $R(1,\eps)$ to distinct points, which are
non-parallel as $\End_K(R)$ has zero radical.
\end{rem}

\begin{nrtxt}
In a first step we describe $\parallel$-iso\-mor\-phisms. Let $R$ and $R'$ be
rings. By formula (14) in \cite[p.~118]{blu+h-03a} we obtain $\#\rad R
=\#\,\cC$ for every parallel class $\cC\subset\PP(R)$. Recall that according to
(\ref{eq:parallel}) the set of parallel classes of $\PP(R)$ can be identified
via (\ref{eq:canon.Pepi}) with $\PP(\overline R)$, where $\overline R=R/\rad
R$. Hence there exists a $\parallel$-iso\-mor\-phism $\PP(R)\to\PP(R')$ if, and
only if,
\begin{equation}\label{eq:exist-parallel}
    \#\rad R =\#\rad R' \mbox{ and } \#\,\PP(\overline R) = \#\,\PP(\overline{R'}).
\end{equation}
Provided that these conditions are satisfied, the description of all
$\parallel$-iso\-mor\-phisms $\PP(R)\to\PP(R')$ is a trivial task: If we fix
one $\parallel$-iso\-mor\-phism $\phi_0:\PP(R)\to\PP(R')$, then every
$\parallel$-iso\-mor\-phism $\phi:\PP(R)\to\PP(R')$ has the form
$\alpha\phi_0$, where $\alpha$ belongs to the group of
$\parallel$-auto\-mor\-phisms of $\PP(R)$. This group in turn is isomorphic
to the wreath product of the symmetric group of an arbitrarily chosen
parallel class $\cC_0\subset\PP(R)$ and the symmetric group of $\PP(\overline
R)$.

For each $\parallel$-iso\-mor\-phism $\phi:\PP(R)\to\PP(R')$ the mapping
\begin{equation}\label{def:phiquer}
 \overline\phi:\PP(\overline R)\to\PP(\overline {R'}):\overline p
  \mapsto\overline{p^\phi}
\end{equation}
is a well defined bijection. Observe that \emph{every\/} bijection
$\PP(\overline R)\to\PP(\overline{R'})$ arises in this way. The following is an
immediate consequence of Corollary~\ref{cor:Rquer} (\ref{Rquer.a}):
\end{nrtxt}

\begin{prop}\label{prop:morphRquer}
Let $R$ and $R'$ be rings. A $\parallel$-iso\-mor\-phism
$\phi:\PP(R)\to\PP(R')$ is a $\dis$-iso\-mor\-phism if, and only if, the
corresponding mapping $\overline\phi$, defined in \gerade{(\ref{def:phiquer})},
is a $\dis$-iso\-mor\-phism.
\end{prop}

\begin{nrtxt}
Proposition~\ref{prop:morphRquer} implies that it suffices to study the
$\dis$-iso\-mor\-phisms $\PP(\overline R)\to\PP(\overline {R'})$. However,
given a $\dis$-iso\-mor\-phism $\phi:\PP(R)\to\PP(R')$, the mapping
$\overline\phi$ only will tell us how $\phi$ acts on the \emph{set of all
parallel classes}. Unless the radical of $R$ is zero, $\overline \phi$ contains
no information at all about the action of $\phi$ on any \emph{parallel class}.
Every bijection $\beta : \PP(R)\to\PP(R)$ which fixes all parallel classes (as
sets) is a $\dis$-auto\-mor\-phism, with $\overline\beta$ the identity on
$\PP(\overline R)$. Therefore, in general, we cannot expect to describe $\phi$
``algebraically''. However, as above, we have that the group of
$\dis$-auto\-mor\-phisms of $\PP(R)$ is isomorphic to the wreath product of the
symmetric group of a parallel class $\cC_0$ and the group of
$\dis$-auto\-mor\-phisms of $\PP(\overline R)$.

We refer to \cite[pp.~359--360]{bart+b-85} for an example of a
$\dis$-auto\-mor\-phism which cannot be described algebraically even though the
Jacobson radical is zero: It is based upon the projective line over a
polynomial ring $K[X]$, where $K$ is a commutative field with characteristic
$\neq 2$. Note that such a ring does not admit proper Jordan endomorphisms
\cite[pp.~2--3]{jac-68}.
\end{nrtxt}

\begin{nrtxt}\label{:semilokal}
A ring $R$ is called \emph{semilocal\/} if $R/\rad R$ is artinian. Since
$\rad(R/\rad R)=\{0\}$, this means that R/rad R is \emph{semisimple}, i.e.
artinian with zero radical. By the Wedderburn-Artin theorem, a ring is
semisimple if, and only if, it is isomorphic to a direct product of finitely
many matrix rings over fields. Note that according to \cite[Sec.~2]{veld-95}
each semilocal ring has stable rank $2$, and observe that in \cite{herz-95}
and \cite{veld-95} the terminology is different, as semilocal rings are
called \emph{semiprimary\/} there.

The results of Sections~\ref{se:endo-ringe} and \ref{se:produkte} together with
Proposition~\ref{prop:morphRquer} will give a complete description of the
$\dis$-iso\-mor\-phisms for projective lines over semilocal rings. See
Corollary~\ref{cor:iso}. We shall work with endomorphism rings of
finite-dimensional vector spaces rather than matrix rings.
\end{nrtxt}

\section{Endomorphism rings.}\label{se:endo-ringe}

\begin{nrtxt}
We now study $\dis$-iso\-mor\-phisms between the projective lines over
endomorphism rings $R=\End_K(\vU)$ and $R'=\End_{K'}(\vU')$ of vector spaces
$\vU$ and $\vU'$. We assume throughout this section that
\begin{equation}\label{eq:dim-voraus}
     1\leq \dim \vU<\infty\mbox{ and }1\leq \dim \vU'<\infty.
\end{equation}
The sets $\cG$ and $\cG'$ are the point sets of the corresponding Grassmann
spaces $\fM$ and $\fM'$, respectively. By (\ref{def:Psi}), we have bijections
$\Psi:\PP(R)\to\cG$ and $\Psi':\PP(R')\to\cG'$.
\end{nrtxt}

\begin{prop}\label{prop:iso.endo}
Let  $R=\End_K(\vU)$ and $R'=\End_{K'}(\vU')$, and  let $\phi:\PP(R)\to\PP(R')$
be a mapping. Then the following statements are equivalent:

\vspace{-0.9ex}\begin{enumerate}\itemsep0pt

\item\label{iso.endo.a} $\phi$ is a $\dis$-iso\-mor\-phism.

\item\label{iso.endo.b} $\phi$ is a $\sim$-iso\-mor\-phism.

\item\label{iso.endo.c} $\Psi^{-1}\phi\Psi':\cG\to\cG'$ is a collineation
$\fM\to\fM'$.
\end{enumerate}
\end{prop}

\begin{txt}
We omit the proof, because this proposition is just that particular case of
Proposition~\ref{prop:iso.prod} where $m=m'=1$.
\end{txt}

\begin{nrtxt}
In the following theorem we use the terminology of \ref{:beispiele}.
Moreover, we introduce the following notations and conventions: The dual space
of $\vU$, which is a right vector space over $K$, is denoted by $\widehat \vU$.
We use the symbol $\langle\cdot,\cdot\rangle$ for the canonical pairing
$\vU\times\widehat\vU\to K$. For each $a\in\End_K(\vU)$ the transpose mapping
is written as $a^\T$; it is an element of $\End_K(\widehat \vU)$. The elements
of $\widehat \vU\times \widehat \vU$ are considered as columns. Each such
column $(\widehat \vv,\widehat \vw)^\T$ acts as a linear form on $\vU\times
\vU$ via the formal matrix product
\begin{equation}
    (\vv,\vw)\mapsto (\vv,\vw)\Mat1{\!\!\widehat \vv\!\!\\ \!\!\widehat \vw \!\!}
    :=\langle \vv,\widehat\vv\rangle + \langle \vw,\widehat\vw\rangle.
\end{equation}
This allows us to identify the dual of $\vU\times \vU$ with $\widehat \vU\times
\widehat \vU$.
\end{nrtxt}

\begin{thm}\label{thm:alg.darst}
Let  $R=\End_K(\vU)$ and $R'=\End_{K'}(\vU')$. A bijection
$\phi:\PP(R)\to\PP(R')$ is a $\dis$-iso\-mor\-phism if, and only if, one of the
following conditions is satisfied:

\vspace{-0.9ex}\begin{enumerate}\itemsep0pt
 \item\label{alg.darst.a}
$\dim \vU=\dim \vU'=1$.

 \item\label{alg.darst.b}
 $\dim \vU>1$ and $\phi=\widetilde\alpha\widetilde\gamma$, where $\alpha:R\to R'$ is
either an isomorphism or an anti-iso\-mor\-phism, and $\gamma\in\GL_2(R')$.
\end{enumerate}\vspace{-0.9ex}
Moreover, for\/ $\dim \vU>1$, the mapping $\alpha$ can be written either as
$x\mapsto h^{-1}xh$ with a semilinear bijection $h:\vU\to \vU'$, or as
$x\mapsto h^{-1}x^\T h$ with a semilinear bijection $h:\widehat \vU\to \vU'$,
respectively.
\end{thm}

\begin{proof}
Suppose that (\ref{alg.darst.a}) holds. Then $R\cong K$ and $R'\cong K'$ so
that $\phi$ is a $\dis$-iso\-mor\-phism. If $\phi$ is given as in
(\ref{alg.darst.b}) then it is a $\dis$-iso\-mor\-phism according to
\ref{:beispiele}.

Conversely, let $\phi$ be a $\dis$-iso\-mor\-phism. We read off from
\cite[Thm.~4.4]{blu+h-02z} that $\dim \vU=\dim \vU'$, whence we can restrict
ourselves to the case $\dim \vU>1$.

By \cite[Thm.~4.4~(c)]{blu+h-02z}, applied to $\Psi^{-1}\phi\Psi'$, there are
two mutually exclusive possibilities:

(i) There is a semilinear bijection $f:\vU\times \vU \to \vU'\times \vU'$
such that for all $\vX\in\cG$ the image of $\vX$ under ${\Psi^{-1}\phi\Psi'}$
equals $\vX^f$.

Choose any semilinear bijection $h:\vU\to \vU'$ with the same accompanying
isomorphism $K\to K'$ as $f$. Then the mapping $f$ can be written formally as
\begin{equation}
    (\vu_1,\vu_2)\stackrel{f}\longmapsto
    (\vu_1^h,\vu_2^h)\Mat2{g_{11} & g_{12}\\g_{21} & g_{22}},
\end{equation}
with a matrix $\gamma:=(g_{ij})\in\GL_2(R')$; cf.\
\cite[pp.~642--643]{lang-95}. This $\gamma$ is at the same time a linear
bijection of $\vU'\times\vU'$; cf.\ \ref{:adjacent}. Clearly, $\alpha: R\to R':
x\mapsto h^{-1}xh$ is an isomorphism of rings. A straightforward calculation
yields
\begin{equation}
    (\vU^{(a,b)})^f= (\vU'^{(a^\alpha,b^\alpha)})^\gamma
    \mbox{ for all } (a,b)\in R^2 \mbox{ with } R(a,b)\in\PP(R).
\end{equation}
Consequently, $R(a,b)^\phi = R(a,b)^{\widetilde\alpha\widetilde\gamma}$, as
required.

(ii) There is a semilinear bijection $f:\widehat \vU\times \widehat \vU \to
\vU'\times \vU'$ such that for all $\vX\in\cG$ the image of $\vX$ under
${\Psi^{-1}\phi\Psi'}$ equals $(\vX^\perp)^f$, where $\vX^\perp$ denotes the
annihilator of $\vX$.

Let $R(a,b)\in\PP(R)$ be a point. There are $c,d\in R$ such that $M:=\SMat2{a &
b\\c&d}$ is invertible. Let $(v,w)^\T$ be the second column of $M^{-1}$. We
claim that
\begin{equation}\label{eq:U_ortho}
\vU^{(a,b)\perp}=\left\{ \Mat1{\!\!\widehat\vu^{v^\T}\!\! \\
                                \!\!\widehat\vu^{w^\T}\!\!}
                          \mid \widehat\vu\in\widehat\vU\right\}.
\end{equation}
Observe that for all $(x,y)\in R^2$, for all $\vu\in\vU$, and all
$\widehat\vu\in\widehat\vU$ we have
\begin{equation}\label{eq:ortho=}
    (\vu^x,\vu^y)\Mat1{\!\!\widehat\vu^{v^\T}\!\!
    \\
    \!\!\widehat\vu^{w^\T}\!\!} =
    \langle \vu^x,\widehat\vu^{v^\T}\rangle +
    \langle \vu^y,\widehat\vu^{w^\T}\rangle
    =\langle \vu^{xv+yw},\widehat\vu\rangle.
\end{equation}
Let us write $\widehat\vY$ for the set on the right hand side of
(\ref{eq:U_ortho}). Equation (\ref{eq:ortho=}) for $(x,y):=(a,b)$ says that
$\vU^{(a,b)}\subseteq\widehat\vY{}^\perp$, since $av+bw=0$. Likewise,
(\ref{eq:ortho=}) for $(x,y):=(c,d)$ gives the reverse inclusion as follows: We
have $cv+dw=1$. Thus for each non-trivial linear form in $\widehat\vY$ there is
a non-zero vector of $\vU^{(c,d)}$ which is not in its kernel. This means that
$\widehat\vY{}^\perp$ is contained in some complement of $\vU^{(c,d)}$ and, by
the above, this complement has to be $\vU^{(a,b)}$. Altogether, we get
$\widehat\vY{}^\perp=\vU^{(a,b)}$ which is equivalent to (\ref{eq:U_ortho}).

Choose any semilinear bijection $h:\widehat\vU\to \vU'$ with the same
accompanying anti-iso\-mor\-phism $K\to K'$ as $f$. Then the mapping $f$ can
be written as
\begin{equation}\label{eq:anti}
    \Mat1{\!\!\widehat\vu_1\!\! \\ \!\!\widehat\vu_2\!\!}\stackrel{f}\longmapsto
    (-\widehat\vu_2^h,\widehat\vu_1^h)\Mat2{g_{11} & g_{12}\\g_{21} &
    g_{22}},
\end{equation}
with a matrix $\gamma:=(g_{ij})\in\GL_2(R')$. Clearly, $\alpha: R\to R':
x\mapsto h^{-1}x^\T h$ is an anti-iso\-mor\-phism of rings. Now it is
straightforward to show that $R(a,b)^\phi =
R'(-w^\alpha,v^\alpha)^{\widetilde\gamma}=R(a,b)^{\widetilde\alpha\widetilde\gamma}$.
\end{proof}

\begin{txt}
We refer to \cite[Thm.~5.1]{blu+h-00c} for a similar result for the particular
case that $\dim\vU=\dim\vU'=2$, where $\alpha$ is only assumed to be a
bijective $\dis$-mor\-phism. (The description based on formula (9) in that
theorem is erroneous; a correct version can be derived from (\ref{eq:anti})
above.)
\end{txt}

\begin{cor}\label{cor:jordan}
Let $R$ be the ring of $n\times n$ matrices over a field $K$, and let $R'$ be
the ring of $n'\times n'$ matrices over a field $K'$. Moreover, let $n>1$,
and let $\omega:R\to R'$  be a Jordan iso\-mor\-phism. Then the following
hold:

\vspace{-0.9ex}\begin{enumerate}\itemsep0pt

\item\label{jordan.a} $n=n'$ and $\omega$ is either an isomorphism or an
anti-iso\-mor\-phism.

\item\label{jordan.b} If  $\omega$ is an isomorphism, then there are an
isomorphism $\beta:K\to K'$ and a matrix $G\in R'^*$ such that
$X^\omega=G^{-1}X^\beta G$ holds for all $X\in R$.

\item\label{jordan.c} If  $\omega$ is an anti-iso\-mor\-phism, then there are
an anti-iso\-mor\-phism $\beta:K\to K'$ and a matrix $G\in R'^*$ such that
$X^\omega=G^{-1}(X^\beta)^\T G$ holds for all $X\in R$.
\end{enumerate} \vspace{-0.9ex}
\gerade(In cases \gerade{(\ref{jordan.b})} and \gerade{(\ref{jordan.c})} the
mapping $\beta $ acts on the entries of the matrix $X$.\gerade)
\end{cor}
\begin{proof}
Choose any matrix $X\in R$. By (\ref{eq:bart2}), the Jordan isomorphism
$\omega$ induces a $\dis$-iso\-mor\-phism
$\widetilde\omega:\PP(R)\to\PP(R')$. Putting $a=1\in R$ and $b=X+1\in R$ in
(\ref{eq:bart2}) gives
\begin{equation}\label{eq:sigma}
    R(X,1)^{\widetilde\omega}=R'(X^\omega,1).
\end{equation}
We repeat the proof of Theorem~\ref{thm:alg.darst} for $\phi=\widetilde\omega$,
$\vU=K^n$, $\vU\times\vU=K^{2n}$ etc. This gives $n=n'>1$ and there are
accordingly two cases:

In case (i) let $\beta : K\to K'$ be the accompanying automorphism of $f$. We
may choose $h$ in such a way that its matrix is $1\in R'$. Hence $X^\alpha =
X^\beta$, where $\beta$ acts on the entries of $X$. Since $\widetilde\alpha$
and $\widetilde\omega$ map $R(1,0)$, $R(0,1)$, and $R(1,1)$ to $R'(1,0)$,
$R'(0,1)$, and $R'(1,1)$, respectively, we conclude that $\gamma=\diag(G,G)$
for a matrix $G\in R'^*$. This gives
\begin{equation}\label{eq:alpha-gamma}
 R(X,1)^{\widetilde\omega}=R(X,1)^{\widetilde\alpha\widetilde\gamma}=R'(X^\beta G,G)=R'(G^{-1}X^\beta G,1).
\end{equation}
From (\ref{eq:sigma}) and (\ref{eq:alpha-gamma}) follows
$X^\omega=G^{-1}X^\beta G$. So $\omega$ is an isomorphism of rings.

In case (ii) it can be shown similarly that $\omega$ is an anti-iso\-mor\-phism
with the required properties: In order to calculate $R(X,1)^{\widetilde\alpha}$
according to \ref{:beispiele}~(c) one may use the matrix $M:=\SMat2{X &1\\1&0}$
so that $M^{-1}=\SMat2{0 & 1\\  1&-X}$.
\end{proof}

\begin{txt}
The statement of Corollary~\ref{cor:jordan} is well known: See
\cite[Thm.~3.24]{wan-96}, where Jordan isomorphisms are called
\emph{semi-iso\-mor\-phisms\/} instead. Of course, the assertion in
(\ref{jordan.a}) is also valid when $n=1$, provided that the word ``either'' is
deleted from the text: In fact, then $n'=1$, too, and we have a Jordan
isomorphism of fields, which is an isomorphism or anti-iso\-mor\-phism by Hua's
theorem. See \cite[Thm.~2.25]{wan-96}.
\end{txt}

\section{Direct products.}\label{se:produkte}

\begin{nrtxt}
Now we consider direct products of rings and the associated projective lines.
Let $R=\prod_{i\in I} R_i$ be the direct product of a family $(R_i)_{i\in I}$
of rings, where it is tacitly assumed throughout this section that $I$ is
non-empty. Then $\GL_2(R)\cong \prod_{i\in I}\GL_2(R_i)$, whence the points
of $\PP(R)$ are exactly the submodules $R\big((a_i)_{i\in I},(b_i)_{i\in
I}\big)\le R^2$ with $R_i(a_i,b_i)\in\PP(R_i)$ for all $i\in I$. So we can
identify each $p\in\PP(R)$ with a family  $(p_i)_{i\in I}$ in the direct
product (cartesian product) $\prod_{i\in I}\PP(R_i)$ of the projective lines
$\PP(R_i)$, and vice versa.

As a general rule, we adopt the following notation: Given an $a\in R$ and an
index $j\in I$ we write $a_j$ for the component of $a$ in $R_j$, i.e., we
assume that $a=(a_i)_{i\in I}$. The same kind of notation is used whenever
applicable, e.g. for points $p\in\PP(R)$. Observe that, as before, we use the
same symbols ($\dis$, $\parallel$, $\sim$) for the corresponding relations on
all our projective lines.
\end{nrtxt}

\begin{prop}\label{prop:rel.prod}
Let $R=\prod_{i\in I} R_i$. Identify $\PP(R)$ with $ \prod_{i\in I}\PP(R_i)$ as
above. Then the following statements hold for all $p$, $q$, and $r\in
\PP(R)$\gerade:

\vspace{-0.9ex}\begin{enumerate}\itemsep0pt
 \item\label{rel.prod.a}
$p\dis q$ $\Leftrightarrow $ $\forall\,i\in I: p_i\dis  q_i$.
 \item\label{rel.prod.b}
$p\parallel q$ $\Leftrightarrow $ $\forall\,i\in I: p_i\parallel q_i$.
 \item\label{rel.prod.c}
$p\sim q$ via $r$ $\Leftrightarrow $ $\exists\, j\in I: \big(p_j\sim q_j$ via
$r_j$ and $\forall\,i\in I\setminus \{j\}: p_i\parallel q_i\parallel r_i$\big).
\end{enumerate}
\end{prop}

\begin{proof}
(\ref{rel.prod.a}): This is clear from $\GL_2(R)\cong \prod_{i\in
I}\GL_2(R_i)$.

(\ref{rel.prod.b}): This is immediate from (\ref{rel.prod.a}) and the
definition of parallel points.

(\ref{rel.prod.c}): Let $p\sim q$ via $r$. Since $r\notpar p$, by
(\ref{rel.prod.b}) there is at least one $j\in I$ such that $r_j\notpar p_j$.
This means that there is a point $y_j\in\PP(R_j)$ with  $y_j\dis r_j$ and
$y_j\notdis p_j$.

We show first that $r_i\parallel q_i$ if $i\neq j$:

For each $i\neq j$ consider an arbitrary point $y_i\in\PP(R_i)$ such that
$y_i\dis r_i$; at least one such $y_i$ exists. Then $y:=(y_i)_{i\in I}\dis r$,
whence $y\dis p$ or $y\dis q$ must hold by assumption. But $y\dis p$ is
impossible since $y_j\notdis p_j$. So $y\dis q$, and in particular $y_i\dis
q_i$ for each $i\neq j$. Altogether $r_i\parallel q_i$ for $i\neq j$, as
desired.

Since $r\notpar q$, we conclude from (\ref{rel.prod.b}) and the above that
$r_j\notpar q_j$ must hold. As before, we can now infer that $r_i\parallel p_i$
for $i\neq j$.

It remains to show, for all $x_j\in\PP(R_j)$, that $x_j\dis  r_j$ implies
$x_j\dis p_j$ or $x_j\dis q_j$. For each $i\neq j$ choose an $x_i\in\PP(R_i)$
with $x_i\dis r_i$. Then $x:=(x_i)_{i\in I}\dis r$, and (\ref{rel.prod.a})
gives the assertion.

Conversely, let $p_j\sim q_j$ via $r_j$ for a $j\in I$ and $p_i\parallel
q_i\parallel r_i$ for all $i\neq j$. By (\ref{rel.prod.b}), we have $r\notpar
p,q$. Let $x\dis r$. Then (\ref{rel.prod.a}) and the assumption yield $x_j\dis
p_j$ or $x_j\dis q_j$. On the other hand, for $i\neq j$, we get from
(\ref{rel.prod.a}) that $x_i\dis p_i$, since $r_i\parallel p_i$, and likewise
$x_i\dis  q_i$, since $r_i\parallel q_i$. So $x\dis p$ or $x\dis q$.
\end{proof}

\begin{txt}
Note that statement (\ref{rel.prod.b}) reflects the algebraic fact
$\rad\left(\prod_{i\in I} R_i\right)=\prod_{i\in I} \rad R_i$ about Jacobson
radicals; cf.\ \ref{para:parallel+adj}.
\end{txt}

\begin{nrtxt}
Let now $(\fM_i)_{i\in I}$ be a family of partial linear spaces
$\fM_i=(\cP_i,\cL_i)$.  The \emph{direct product\/} (or \emph{Segre product\/})
of the $\fM_i$ is the partial linear space $\fM:=\prod_{i\in
I}\fM_i:=(\cP,\cL)$ with point set $\cP:=\prod_{i\in I} \cP_i$ and line set
$\cL:=\bigcup_{i\in I}\cL_{(i)}$, where for each $j\in I$ we define
\begin{equation}\label{def:prod.gerade}
\cL_{(j)}:=  \big\{
             \{x\in\cP\mid x_j\in L_j \wedge
                           \forall\,i\in I\setminus\{j\}: x_i=p_i \}
             \mid L_j\in\cL_j , p\in \cP \big\}.
\end{equation}
See \cite{nau+p-01}. Note that in this definition of a line, the point
$p_j\in\PP(R_j)$ is irrelevant.
\end{nrtxt}

\begin{nrtxt}
Let $R=\prod_{i\in I} R_i$, where each $R_i$ is an endomorphism ring
$\End_{K_i}(\vU_{i})$ of a left vector space over a field $K_i$,
$1\leq\dim\vU_{i}<\infty.$ Let $\fM_i=(\cG_i,\cL_i)$ be the Grassmann space
associated to $\PP(R_i)$, and let $\Psi_i:\PP(R_i)\to\cG_i$ be given as in
(\ref{def:Psi}). We call
\begin{equation}
   \fM:=(\cG,\cL):=\prod_{i\in I}\fM_i.
\end{equation}
the \emph{product space\/} associated to $\PP(R)$. Then, generalizing
(\ref{def:Psi}), the mapping
\begin{equation}\label{eq:prod.Psi}
    \Psi : \PP(R)\to \cG :
    p \mapsto (p_i^{\Psi_i})_{i\in I}
\end{equation}
is a bijection. According to our notational setting we have the trivial
identity $(p^\Psi)_i=p_i^{\Psi_i}$ for all $i\in I$. We are now in a position
to generalize Proposition~\ref{prop:adjazent} as follows:
\end{nrtxt}

\begin{prop}\label{prop:adj.prod}
Let $R=\prod_{i\in I} R_i$, where $R_i=\End_{K_i}(\vU_{i})$ as above. Let $\fM$
be the product space associated to $\PP(R)$. Then the following statements hold
for all $p$, $q$, and $r\in\PP(R)$, where $\Psi$ is given as in
\gerade{(\ref{eq:prod.Psi}):}

\vspace{-0.9ex}\begin{enumerate}\itemsep0pt
 \item\label{adj.prod.a}
    $p\sim q  \Leftrightarrow p^\Psi$ and
    $q^\Psi$  are distinct collinear points of $\fM$.

 \item\label{adj.prod.b}
$p\sim q$  via $r \Leftrightarrow p^\Psi$, $q^\Psi$, and $r^\Psi$ are distinct
collinear points of $\fM$.
\end{enumerate}
\end{prop}
\begin{proof}
It suffices to prove (\ref{adj.prod.b}), since every line of $\fM$ contains at
least three points. By Proposition~\ref{prop:rel.prod} (\ref{rel.prod.c}), we
have that $p\sim q$ via $r$ is equivalent to
\begin{equation}
\exists\, j\in I: \big(p_j\sim q_j \mbox{ via } r_j \mbox{ and } \forall\,i\in
I\setminus \{j\}: p_i= q_i=r_i\big).
\end{equation}
(Recall that here~$\parallel$ is equality.) By Proposition~\ref{prop:adjazent}
(\ref{adjazent.b}), this is equivalent to
\begin{equation}
\exists\, j\in I: \left\{\begin{array}{l}
  (p^\Psi)_j, (q^\Psi)_j, (r^\Psi)_j
  \mbox{ are distinct collinear points of }\fM_j
  \\
  \mbox{and }\forall\,i\in I\setminus \{j\}: (p^\Psi)_i=(q^\Psi)_i=(r^\Psi)_i.
\end{array}\right.
\end{equation}
Finally, by (\ref{def:prod.gerade}), this is means that $p^\Psi$, $q^\Psi$, and
$r^\Psi$ are distinct collinear points of $\fM$.
\end{proof}

\begin{txt}
In the following results we confine ourselves to finite (non-empty) products.
In case of an infinite product, the space $\fM$ is not connected, and our
proofs for Proposition~\ref{prop:iso.prod}, Theorem~\ref{thm:alg.darst.prod},
and Theorem~\ref{thm:NP} are not applicable. So the rings in the following
proposition and the subsequent theorem are exactly the semisimple rings.
\end{txt}

\begin{prop}\label{prop:iso.prod}
Let $R=\prod_{i=1}^mR_i$ and $R'=\prod_{j=1}^{m'}R'_j$ with
$R_i=\End_{K_i}(\vU_{i})$, $R_j'=\End_{K'_j}(\vU'_{j})$, where $1\leq
\dim\vU_{i},\;\dim\vU'_{j}<\infty$. Let $\fM$ and $\fM'$ be the associated
product spaces and suppose that $\Psi$ and $\Psi'$ are given according to
\gerade{(\ref{eq:prod.Psi})}. For a mapping $\phi:\PP(R)\to\PP(R')$, the
following statements are equivalent:

\vspace{-0.9ex}
\begin{enumerate}\itemsep0pt

\item\label{iso.prod.a} $\phi$ is a $\dis$-iso\-mor\-phism.

\item\label{iso.prod.b} $\phi$ is a $\sim$-iso\-mor\-phism.

\item\label{iso.prod.c} $\Psi^{-1}\phi\Psi':\cG\to\cG'$ is a collineation
$\fM\to\fM'$.
\end{enumerate}
\end{prop}

\begin{proof}
(\ref{iso.prod.a}) $\Rightarrow $ (\ref{iso.prod.c}): By definition, for all
$p$, $q$ and $r$ in $\PP(R)$ we have $p\sim q$ via $r$ if, and only if,
$p^\phi\sim q^\phi$ via $r^\phi$. Recall that every line of $\fM$ has at least
three distinct points. The same applies to $\PP(R')$ and $\fM'$. Hence the
characterization of collinear points in Proposition~\ref{prop:adj.prod}
(\ref{adj.prod.b}) gives the required result.

(\ref{iso.prod.c}) $\Rightarrow $ (\ref{iso.prod.b}): This is immediate from
Proposition~\ref{prop:adj.prod} (\ref{adj.prod.a}).

(\ref{iso.prod.b}) $\Rightarrow $ (\ref{iso.prod.a}): For each
$i\in\{1,2,\ldots,m\}$ let us consider $\PP(R_i)$ as the set of vertices of the
so-called \emph{adjacency graph}, where two vertices are joined by an edge if,
and only if, they are adjacent. The \emph{distance function\/} in this graph
will be written as $\dist$. By Proposition~\ref{prop:adjazent}
(\ref{adjazent.a}), the bijection $\Psi_i$ is an isomorphism of the adjacency
graph on $\PP(R_i)$ onto the Grassmann graph on $\cG_i$. A simple induction
shows that
\begin{equation}\label{eq:dist.i}
     \forall\, p_i,q_i\in\PP(R_i) :     \dist(p_i,q_i)=
     \dim \vU_{i}-\dim(p_i^{\Psi_i}\cap q_i^{\Psi_i}) \leq \dim\vU_{i}.
\end{equation}
The maximal distance $\dim\vU_{i}$ is assumed precisely when $p_i\dis q_i$.

Similarly, $\PP(R)$ and its adjacency relation give rise to an adjacency graph,
and (\ref{eq:dist.i}), together with Proposition~\ref{prop:rel.prod}
(\ref{rel.prod.c}), yields the following formula:
\begin{equation}\label{eq:dist}
  \forall\, p,q\in\PP(R) : \dist(p,q)= \sum_{i=1}^m \dist(p_i,q_i)
  \leq \sum_{i=1}^m\dim\vU_{i}.
\end{equation}
By Proposition~\ref{prop:rel.prod} (\ref{rel.prod.a}), here the points $p$ and
$q$ are distant, if, and only if, their distance attains the bound given in
(\ref{eq:dist}).

Let now $\phi$ be a $\sim$-iso\-mor\-phism. Then two points of $\PP(R)$ with
maximal distance (in the adjacency graph) go over to points of $\PP(R')$ with
maximal distance and vice versa. By the above, this means that $\phi$ is a
$\dis$-iso\-mor\-phism.
\end{proof}

\begin{txt}
This brings us to our main result:
\end{txt}

\begin{thm}\label{thm:alg.darst.prod}
Let $R=\prod_{i=1}^mR_i$ and $R'=\prod_{j=1}^{m'}R'_j$ with
$R_i=\End_{K_i}(\vU_{i})$, $R_j'=\End_{K'_j}(\vU'_{j})$, where $1\leq
\dim\vU_{i},\;\dim\vU'_{j}<\infty$. For a mapping $\phi:\PP(R)\to\PP(R')$, the
following statements are equivalent:

\vspace{-0.9ex}\begin{enumerate}\itemsep0pt
 \item\label{alg.darst.prod.a}
$\phi$ is a $\dis$-iso\-mor\-phism.

 \item\label{alg.darst.prod.b}
$m=m'$, and there is a permutation $\sigma$ of $\{1,2,\ldots,m\}$ such that
for each $p\in\PP(R)$ and each $k\in\{1,2,\ldots,m\}$ we have
$(p^\phi)_{k^\sigma}=(p_k)^{\phi_k}$ for $\dis$-iso\-mor\-phisms
$\phi_k:\PP(R_k)\to\PP(R'_{k^\sigma})$.
\end{enumerate}
In this case, if $\dim\vU_{i}>1$ for all $i\in\{1,2,\ldots,m\}$, then
$\phi=\widetilde\alpha\widetilde\gamma$, where $\gamma\in\GL_2(R')$ and
$\alpha:R\to R'$ is a Jordan isomorphism such that for each $x\in R$ and each
$k\in\{1,2,\ldots,m\}$ we have $(x^\alpha)_{k^\sigma}=(x_k)^{\alpha_k}$, with
$\alpha_k:R_k\to R'_{k^\sigma}$ an isomorphism or an anti-iso\-mor\-phism.
\end{thm}
\begin{proof}
(\ref{alg.darst.prod.a}) $\Rightarrow $ (\ref{alg.darst.prod.b}): By
Proposition~\ref{prop:iso.prod}, the mapping $\Psi^{-1}\phi\Psi':\cG\to\cG'$ is
a collineation $\fM\to\fM'$. Since each $\fM_i$ is strongly connected and
contains at least one line, the assertion follows from Theorem~\ref{thm:NP} in
the Appendix and Proposition~\ref{prop:iso.endo}.

While (\ref{alg.darst.prod.b}) $\Rightarrow $ (\ref{alg.darst.prod.a}) is
obvious, the other assertions follow from Theorem~\ref{thm:alg.darst}.
\end{proof}

\begin{txt}
From this and Proposition~\ref{prop:morphRquer} we obtain the following:
\end{txt}

\begin{cor}\label{cor:iso}
Let $R$ and $R'$ be semilocal rings. Then a bijection $\phi:\PP(R)\to\PP(R')$
is a $\dis$-iso\-mor\-phism if, and only if, it is a
$\parallel$-iso\-mor\-phism such that the induced bijection
$\overline\phi:\PP(\overline R)\to\PP(\overline{R'})$ \gerade(see
\gerade{(\ref{def:phiquer}))} is a mapping as in
Theorem~\gerade{\ref{thm:alg.darst.prod}}.
\end{cor}

\begin{txt}
From Theorem~\ref{thm:alg.darst.prod}, Corollary~\ref{cor:jordan}
(\ref{jordan.a}) and Hua's theorem \cite[Thm.~2.25]{wan-96}, or from
Theorem~\ref{thm:alg.darst.prod} and \cite[Thm.~3.24]{wan-96}, we infer the
following algebraic description of Jordan isomorphisms between semisimple
rings.
\end{txt}

\begin{cor}\label{cor:jordan-prod}
Let $R=\prod_{i=1}^mR_i$ and $R'=\prod_{j=1}^{m'}R'_j$ with $R_i$, $R_j'$
matrix rings over fields. Let $\omega:R\to R'$ be a Jordan iso\-mor\-phism.
Then $m=m'$, and there is a permutation $\sigma$ of $\{1,2,\ldots,m\}$ such
that for each $x\in R$ and each $k\in \{1,2,\ldots,m\}$ we have
$(x^\omega)_{k^\sigma}=(x_k)^{\omega_k}$, where $\omega_k:R_k\to R'_{k^\sigma}$
is an isomorphism or an anti-iso\-mor\-phism.
\end{cor}

\begin{txt}
Note that this can also be shown in a purely algebraic way: Just replace
``Theorem~2'' with ``\cite[Thm.~3.24]{wan-96}'' in the proof of
\cite[Thm.~3]{kapl-47}, where a similar result is shown for direct products of
finitely many simple algebras of finite dimension.
\end{txt}

\section{Appendix.}

\begin{nrtxt}
The following is a generalization of \cite[Prop.~1.10]{nau+p-01}, as we
consider not only automorphisms but also isomorphisms of product spaces. For
the reader's convenience we stick close to the notation used in
\cite{nau+p-01}:
\end{nrtxt}

\begin{thm}\label{thm:NP}
Let $\fM_i =(X_i,\cL_i)$, $i\in\{1,2,\ldots,m\}$, and $\fM'_j =(X'_j,\cL'_j)$,
$j\in\{1,2,\ldots,m'\}$ be strongly connected partial linear spaces with at
least one line, where $m,m'\geq 1$. Suppose, furthermore, that
\begin{equation}\label{eq:f}
   f:\prod_{i=1}^m X_i \to \prod_{j=1}^{m'} X'_j
\end{equation}
is a collineation of $\fM:=\prod_{i=1}^m \fM_i$ onto $\fM':=\prod_{j=1}^{m'}
\fM'_j$. Then the following assertions hold:

\vspace{-0.9ex}\begin{enumerate}\itemsep0pt
 \item\label{NP.a} $m=m'$.
 \item\label{NP.b} There exists a permutation $\sigma$ of the set
 $\{1,2,\ldots,m\}$, and for each $k\in\{1,2,\ldots,m\}$ there is a
 collineation $f_k:X_k\to X'_{\sigma(k)}$ such that
 \begin{equation}\label{}
    \big(f(c_1,c_2,\ldots,c_m)\big)_{\sigma(k)} = f_k(c_k)
\end{equation}
for all points $(c_1,c_2,\ldots,c_m)\in \prod_{i=1}^m X_i$.
\end{enumerate}
\end{thm}

\begin{proof} (\ref{NP.a}): Let $S$ and $T$ be strong subspaces of $\fM$.
Following \cite[p.~131]{nau+p-01}, we let $S \approx T$ if, and only if, there
exists a finite sequence $Y_0,Y_1,\ldots,Y_b$ of strong subspaces such that
$S=Y_0$, $T=Y_b$, and $\#(Y_i\cap Y_{i-1})\geq 2$ for all
$i\in\{1,2,\ldots,b\}$. This is an equivalence relation on the set of
\emph{all\/} strong subspaces.

Let us choose a point $p$ of $\fM$. We restrict $\approx$ to the (non-empty)
set of strong subspaces \emph{which contain $p$ and at least one more point}.
By \cite[Lemma~1.5]{nau+p-01}, there are precisely $m$ equivalence classes of
this restricted relation. Similarly, we have a relation $\approx'$ on the set
of all strong subspaces of $\fM'$. There are precisely $m'$ equivalence classes
when $\approx'$ is restricted to the set of strong subspaces of $\fM'$ which
contain a fixed point $p'$ and at least one more point. As $f$ and $f^{-1}$
preserve strong subspaces with more than one point, and because of $S\approx
T\Leftrightarrow f(S)\approx' f(T)$, we finally get $m=m'$.
\par
(\ref{NP.b}): By virtue of (\ref{NP.a}), it is easy to see that
Proposition~1.6, Corollary~1.9, and Proposition~1.10 in \cite{nau+p-01} hold,
mutatis mutandis, even under our assumptions. As a matter of fact, the proofs
given in \cite{nau+p-01} and the ones which are needed now are, up to
notational changes, the same. Therefore the assertion follows.
\end{proof}

\begin{txt}
\textbf{Acknowledgement.} The authors are grateful to Peter \v{S}emrl
(Ljubljana) and Matej Bre\v{s}ar (Maribor) for their kind assistance when
searching for literature on Jordan homomorphisms.
\end{txt}



\begin{txt}\footnotesize
Andrea Blunck\\
Fachbereich Mathematik\\
Universit\"at Hamburg\\
Bundesstra{\ss}e 55\\
D--20146 Hamburg\\
Germany\\
email: andrea.blunck@math.uni-hamburg.de
\end{txt}
\begin{txt}\footnotesize
Hans Havlicek\\
Institut f\"ur Diskrete Mathematik und Geometrie\\
Technische Universit\"at\\
Wiedner Hauptstra{\ss}e 8--10\\
A--1040 Wien\\
Austria\\
email: havlicek@geometrie.tuwien.ac.at
\end{txt}
\end{document}